\documentclass[12pt]{amsart}
\usepackage[dvips]{graphicx}
\usepackage{amsmath,amssymb,amsfonts,amsthm,bbm,mathtools}
\usepackage{todonotes,color}

\newcommand{\sss}   { \scriptscriptstyle } 
\newcommand{\nn}{\nonumber}
\newcommand{\R}{{\mathbb R}}
\newcommand{\Z} {{\mathbb Z}}
\newcommand{\N}{{\mathbb N}}
\newcommand{\mP}{\mathbb{P}}

\newcommand{\E}{\mathbb{E}}

\newtheorem{THM}{Theorem}
\newtheorem{COR}[THM]{Corollary}
\newtheorem{EXA}{Example}

\newtheorem{LEM}[THM]{Lemma}
\newtheorem{PRP}[THM]{Proposition}

\newtheorem{COND}{Condition}

\newcommand{\hlf}{\frac{1}{2}}
\newcommand{\ra}{\rightarrow}

\renewcommand{\to}      {\rightarrow}
\newcommand{\blank}[1]{}

\newcommand{\Qed}{\qed \medskip}

\newcommand{\mc}[1]{\mathcal{#1}}

\newcommand{\smallE}{\scriptstyle \rightarrow}
\newcommand{\smallW}{\scriptstyle \leftarrow}

\newcommand{\ssmallW}{\scriptscriptstyle \leftarrow}
\newcommand{\smallN}{\scriptstyle \uparrow}
\newcommand{\smallS}{\scriptstyle \downarrow}

\newcommand{\smallnw}{\scriptscriptstyle \nwarrow}
\newcommand{\smallsw}{\scriptscriptstyle \swarrow}

\newcommand{\Nwsearrow}{\mathrlap{\nwarrow}{\searrow}}

\newcommand{\OTSP}{\begin{picture}(,)
\put(8.1,1.3){$\uparrow$}
\put(.3,-3.3){$\leftarrow$}
\put(0.2,1.3){$\nwarrow$}
\end{picture}\hspace{.4cm}
}
\newcommand{\smallOTSP}{\begin{picture}(,)
\put(6.2,0){$\smallN$}
\put(.3,-3.3){$\smallW$}
\put(1.7,0){$\smallnw$}
\end{picture}\hspace{.4cm}
}
\newcommand{\smallFSOSP}{\begin{picture}(,)
\put(5.5,2.7){$\smallN$}
\put(5.5,-2.3){$\smallS$}
\put(.9,3.1){$\smallnw$}
\put(.9,-2){$\smallsw$}
\put(0.5,.5){$\ssmallW$}
\end{picture}\hspace{.35cm}
}

\newcommand{\NE}{\begin{picture}(,)
\put(2,-5){$\rightarrow$}
\put(0,.5){$\uparrow$}
\end{picture}\hspace{.5cm}
}
\newcommand{\smallNE}{\begin{picture}(,)
\put(1.5,-3){$\smallE$}
\put(0,.5){$\smallN$}
\end{picture}\hspace{.5cm}
}

\newcommand{\SE}{\begin{picture}(,)
\put(1.5,4.8){$\rightarrow$}
\put(-.5,-.5){$\downarrow$}
\end{picture}\hspace{.5cm}
}

\newcommand{\SW}{\begin{picture}(,)
\put(0,4.8){$\leftarrow$}
\put(8.2,-0.5){$\downarrow$}
\end{picture}\hspace{.5cm}
}

\newcommand{\NW}{\begin{picture}(,)
\put(0,-5){$\leftarrow$}
\put(8,0){$\uparrow$}
\end{picture}\hspace{.5cm}
}

\newcommand{\NSW}{\begin{picture}(,)
\put(7.5,0){$\updownarrow$}
\put(0,0){$\leftarrow$}
\end{picture}\hspace{.5cm}
}

\title[ballistic]{Conditions for ballisticity and \\
invariance principle for 
random walk \\
in non-elliptic random environment.}

\author[Holmes]{Mark Holmes}
\address{Department of Statistics, University of Auckland}
\email{m.holmes@auckland.ac.nz}
\author[Salisbury]{Thomas S. Salisbury} 
\address{Department of Mathematics and Statistics, York University}
\email{salt@yorku.ca}


\begin{document}
\maketitle

\begin{abstract}
We study the asymptotic behaviour of random walks in i.i.d.~non-elliptic random environments on $\Z^d$.  Standard conditions (and proofs) for ballisticity and the central limit theorem require ellipticity. We use oriented percolation  and martingale arguments to give non-trivial local conditions for ballisticity and an annealed invariance principle in the non-elliptic setting.\end{abstract}


\maketitle

\section{Introduction}
A central topic in modern probability and statistical physics is the study of {\em random walks in random media}.  Among the most important classes of such models is the so-called random walk in i.i.d.~random environment on $\Z^d$.  While these models have been studied for decades (see e.g.~\cite{Zeit04}), there are a number of fundamental problems that remain open in dimensions $d\ge 2$.  These problems include providing verifiable conditions under which the random walk is ballistic (and/or have Brownian motion as their scaling limit).  Existing results have largely been restricted to situations where the random environment is {\em elliptic}, i.e.~where steps to all nearest neighbours are possible (see e.g.~\cite{BRS16} and the references therein).  In other contexts (such as random walk on percolation clusters) where ellipticity may or may not be assumed, a crucial ingredient for establishing asymptotic behaviour of the walk is a property called {\em reversibility} (see e.g.~\cite{BBT2016} and the references therein).  Except in trivial cases, random walks in i.i.d.~random environments are {\em not} reversible.

We will study random walks in i.i.d.~random environments (RWRE) that are non-elliptic, such as in the following example.

\medskip

\begin{EXA}[$2$-dimensional orthant model]
\label{exa:NE_SW}
At each site $x\in \Z^2$ independently toss a (possibly biased) coin.  If the toss results in heads (probability $p$), insert one directed edge pointing up $\uparrow$ (to $x+(0,1)$) and one pointing right $\rightarrow$ (to  $x+(0,1)$). Otherwise (probability $1-p$) insert directed edges pointing down $\downarrow$ and left $\leftarrow$ (see Figure \ref{fig:orthant_env}).  Now start a random walk at the origin $o$ that evolves by choosing uniformly from available arrows at its current location.
\end{EXA}

\medskip

\begin{figure}
\begin{center}
\includegraphics[scale=.5]{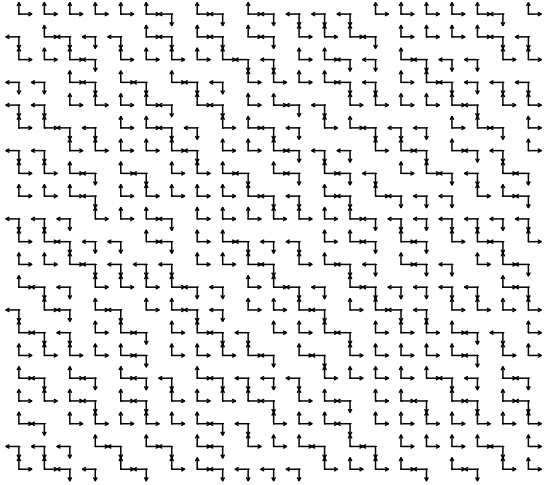} 
\end{center}
\vspace{-1cm}
\caption{A finite region of the random environment in Example~\ref{exa:NE_SW} for $p=.6$. 
}
\label{fig:orthant_env}
\end{figure}

Standard techniques used to establish ballistic behaviour in elliptic environments (see e.g.~\cite{Zeit04}) do not apply to this model (as it is not elliptic!).  
It is proved in \cite{RWDRE} that the random walk in Example \ref{exa:NE_SW} has an asymptotic velocity $v[p]$ that is monotone in $p$, and that $v[\frac12]=0$ by symmetry. It is also established that the walk is transient in direction $\nearrow$ when $p>p_c^{\smallOTSP}$ (where $p_c^{\smallOTSP}$ is the critical probability for oriented site percolation on the triangular lattice on $\Z^2$), using the fact that for such $p$, almost surely the origin is connected to only finitely many sites in direction $\swarrow$. These results do not establish ballisticity (i.e.~that $v[p]$ is non-zero) for any non-trivial value of $p$.

When applied to Example~\ref{exa:NE_SW}, our main results (Theorem \ref{thm:main}, and Propositions \ref{prp:E+} and \ref{prp:transverse}) imply that $v[p]\cdot(1,1)>0$ for $p>p_c^{\smallOTSP}$, where $p_c^{\smallOTSP}\approx 0.5956$ \cite{DeBE,JG}\footnote{The best rigorous bounds that we are aware of are that $.5730<p_c^{\smallOTSP}<0.7491$ \cite{HS_DRE2,BBS}},
and by symmetry that $v[p]\cdot(1,1)<0$ for $p<1-p_c^{\smallOTSP}$.  Moreover in this regime the random walk obeys an invariance principle with deterministic variance, for almost every environment. 

\subsection{The model and main results}
\label{sec:model}

For fixed $d\ge 2$, let $\mc{E}=\{\pm e_i: i=1,\dots,d\}$ be the set of unit vectors in $\Z^d$, and let $\mc{E}_+=\{+e_i:i=1,\dots, d\}$ denote the standard basis vectors.  We use a graphical shorthand for subsets of $\mc{E}$, so that (for example) $\NE=\mc{E}_+$. Let $\mc{P}=M_1(\mc{E})$ denote the set of probability measures on $\mc{E}$, and let $\mu$ be a probability measure on $\mc{P}$.  If $\gamma\in \mc{P}$ we will abuse notation and write $\mu(\gamma)$ for $\mu(\{\gamma\})$.  Let $\Omega=\mc{P}^{\Z^d}$ be equipped with the product measure $\nu=\mu^{\otimes \Z^d}$ (and the corresponding product $\sigma$-algebra).  An environment $\omega=(\omega_x)_{x\in \Z^d}$ is an element of $\Omega$.  We write $\omega_x(e)$ for $\omega_x(\{e\})$. Note that $(\omega_x)_{x\in \Z^d}$ are i.i.d.~with law $\mu$ under $\nu$.  An environment is {\em $2$-valued} if $\mu$ is supported on exactly 2 points (such as in Example \ref{exa:NE_SW}). In this case we take the convention that $\mu$ is supported on $\{\gamma^{\sss(1)},\gamma^{\sss(2)}\}$ with $p=\mu(\gamma^{\sss(1)})$. 

The random walk in environment $\omega$ is a time-homogeneous (but not necessarily irreducible) Markov chain with transition probabilities from $x$ to $x+e$ defined by 
\begin{equation*}
p_{\omega}(x,x+e)=
\omega_x(e).
\end{equation*}
Given an environment $\omega$, we let $\mP_{\omega}$ denote the (quenched) law of this random walk $X_n$, starting at the origin.  Let $P$ denote the law of the annealed/averaged random walk, i.e.~$P(\cdot, \star):=\int_{\star}\mP_{\omega}(\cdot)d\nu$.

For $\gamma\in \mc{P}$, let $\mc{S}(\gamma)=\{e\in \mc{E}:\gamma(e)>0\}\subset \mc{E}$ denote the support of $\gamma$.  For $\mc{A}\subset \mc{E}$ we will write $\mu(\mc{A})$ to mean $\mu(\{\gamma\in \mc{E}:\mc{S}(\gamma)=\mc{A}\})$, i.e.~the $\mu$-measure of the set of probabilities on $\mc{E}$ whose support is $\mc{A}$.  For each environment $\omega$ we let $\mc{G}_x(\omega)=\mc{S}(\omega_x)$ and associate a directed graph $\mc{G}(\omega)$ with vertex set $\Z^d$ and edge set $e(\mc{G})$ given by
\begin{equation*}
(x,x+u)\in e(\mc{G}) \iff u\in\mc{G}_x.
\end{equation*}
Note that under $\nu$, the $(\mc{G}_x)_{x\in \Z^d}$ are i.i.d.~subsets of $\mc{E}$.  The directed graph $\mc{G}(\omega)$ is the entire graph $\Z^d$ (with directed edges), precisely when the environment is {\em elliptic}, i.e.~$\omega_x(u)>0$ for each $u \in \mc{E}, x\in \Z^d$ (i.e.~$\mu(\mc{E})=1$, using our other notation).  Much of the current literature on random walk in random media assumes either (uniform)  ellipticity or reversibility, neither of which hold for Example \ref{exa:NE_SW}.

On the other hand, given a directed graph $\mc{G}=(\mc{G}_x)_{x\in \Z^d}$ (with vertex set $\Z^d$, and such that $\mc{G}_x\ne \varnothing$ for each $x$), we can define a {\em uniform} random environment $\omega=(\omega_x(\mc{G}_x))_{x\in \Z^d}$ on $\mc{G}$. Let $|A|$ denote the cardinality of $A$, and set
\[\omega_x(e)=\begin{cases}
|\mc{G}_x|^{-1}, & \text{ if }e \in \mc{G}_x\\
0, & \text{otherwise}.
\end{cases}\]
The corresponding RWRE then moves by choosing uniformly from available steps at its current location.  
This natural class of RWRE will henceforth be referred to as {\em uniform RWRE}.  
In particular, the 2-dimensional-orthant model (Example \ref{exa:NE_SW}) is the uniform RWRE on the random directed graph which has $\mathcal{G}_x=\NE$ with probability $p$, and $\mathcal{G}_x=\SW$ with probability $1-p$. 

For any $n$, we call a sequence $(y_0, \dots, y_n)$ with each $y_i \in\Z^d$ a \emph{ $\mc{G}$-path} if $(y_i,y_{i+1})\in e(\mc{G})$ for each $i=0,1,\dots, n-1$. For any site $x\in \Z^d$, we let its \emph{forward cluster} $\mathcal{C}_x$ be the set of sites $y\in \Z^d$
such that there exists an $n$ and a $\mc{G}$-path $(y_0, \dots, y_n)$ such that $y_0=x$ and $y_n=y$.  

We say that $V$ is an orthogonal set if $u\cdot v=0$ for every distinct pair $u,v\in V$.  Instead of ellipticity, we will assume the following properties:

\medskip

\begin{COND}
\label{cond:orthogonal}
There exists an orthogonal set $V\subset \mc{E}$ such that $\mu(\mc{S}\cap V\ne \varnothing)=1$.
\end{COND}

\medskip

\begin{COND}
\label{cond:ddim}
There exists an orthogonal set $V'\subset \mc{E}$ with $|V'|=d$ such that $\mu(e\in\mc{S})>0$ for every $e\in V'$.
\end{COND}

\medskip

Condition \ref{cond:orthogonal} requires that there is a set of orthogonal directions such that from any site the walker is able to follow at least one of these directions.  This assumption is precisely that required to ensure that the random walker never gets stuck in a finite set (see \cite[Theorem 1.2]{RWDRE}).  In the presence of Condition \ref{cond:orthogonal}, Condition \ref{cond:ddim} is equivalent to saying that the walk is truly $d$-dimensional.
Note that Example \ref{exa:NE_SW} satisfies Condition \ref{cond:orthogonal} with $V=\{-e_1,e_2\}=\NW$ or equivalently with $V=\{e_1,-e_2\}=\SE$.  It clearly also satisfies Condition \ref{cond:ddim} (with $d=2$).

The RWRE literature contains a number of abstract conditions that imply ballisticity, which we discuss in Section \ref{sec:elliptic}. These can be difficult to verify directly in concrete examples. We turn first to our version of such an abstract condition. Following that we will turn to local conditions, that may be directly verified, and which imply the abstract one. 

Fix $d\ge 2$ and $\ell\in\R^d\setminus\{o\}$. For $\kappa>0$, we consider the cone
$$
\mc{K}_{\kappa,\ell}=\{u\in\R^d: u\cdot\ell\ge \kappa\| u\|\}.
$$
Let $R_n=|\{x: X_m=x \text{ for some }m\le n\}|$ denote the range of the walker up to time $n$.  Our first main result states that, if the forward cluster $\mc{C}_o$ is contained in a cone (whose apex is far from $o$ with only low probability), and the range of the walker is not too small, then the walk is ballistic and satisfies an annealed/averaged invariance principle. 

\medskip

\begin{THM}
\label{thm:main}
Let $d\ge 2$ and assume Conditions \ref{cond:orthogonal} and \ref{cond:ddim}. Let $\alpha, \beta, \kappa>0$ and take $\ell\in\R^d\setminus\{o\}$. Assume the following conditions:
\begin{itemize}
\item[(a)] There exist $C_1,\gamma_1>0$ such that \\
$\nu(\mc{C}_o\subset -n\ell + \mc{K}_{\kappa,\ell})\ge 1-C_1e^{-\gamma_1 n^\beta}$ for all $n \in \N$;
\item[(b)] For every $C>0$, there exist $C_2,\gamma_2>0$ such that \\
$P(R_n\le Cn^{\alpha})\le C_2e^{-\gamma_2 n^\beta}$, for all $n\in \N$.
\end{itemize}
Then there exist $v\in \R^d\setminus \{o\}$ and a non-negative definite 
matrix $\Sigma\in \R^{d\times d}$ such that $P(n^{-1}X_n \ra v)=1$ and under the annealed/averaged measure $P$,
\begin{align}
\Big(\frac{X_{\lfloor nt\rfloor}-vnt}{\sqrt{n}}\Big)_{t\ge 0}\Rightarrow(B_t)_{t\ge 0},\qquad \textrm{ as }n \ra \infty,\nn
\end{align}
where $B_t$ is a $d$-dimensional Brownian motion with covariance matrix $\Sigma$, and $\Rightarrow$ denotes weak convergence. Moreover $v\cdot\ell >0$.
\end{THM}

\medskip

For $\ell'\in \R^d\setminus\{o\}$ set $X'_n=X_n\cdot \ell'$ and call this the \emph{transverse walk}.  In many settings we will be able to conclude that the range of the walker satisfies condition (b) of the theorem by proving that it holds for the range of such a transverse walk.  
For Example \ref{exa:NE_SW}, if we take $\ell=(1,1)$ and $\ell'=(1,-1)$ then the transverse walk is a simple symmetric random walk on $\Z$, so (b) holds (for $\alpha<1/2$). We will show that for $p>p_c^{\smallOTSP}\approx 0.5956$ (a) also holds.  In fact for this model we get a {\em quenched} FCLT by applying \cite[Theorem 1.1]{RS09}.  Our results leave unresolved the question of whether Example \ref{exa:NE_SW} is ballistic when $\frac12<p\le p_c^{\smallOTSP}$, and even whether the speed in direction $(1,1)$ is strictly monotone for $p>p_c^{\smallOTSP}$.  We conjecture that it is strictly monotone in $p\in[0,1]$ (see Figure \ref{fig:vorthant}). When $p=\frac12$, we conjecture that infinitely many sites are visited infinitely often by the walk. 

\begin{figure}
\includegraphics[scale=.5]{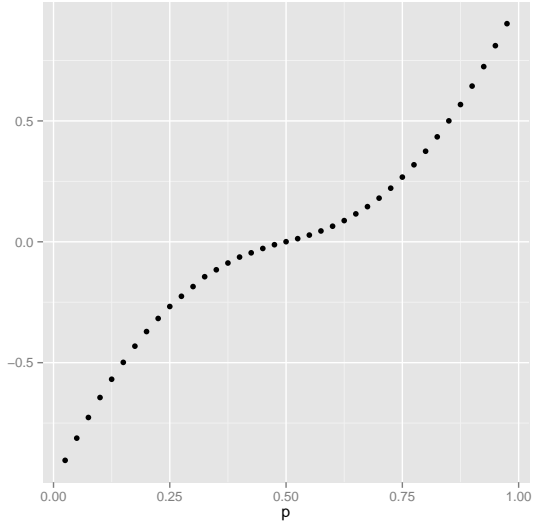}
\caption{Estimates of $v[p]\cdot (1,1)$ as a function of $p$ for the 2-dimensional orthant model (Example \ref{exa:NE_SW}) based on 1000 simulations of 1000 step walks for each $p$.}
\label{fig:vorthant}
\end{figure}
Note that Theorem \ref{thm:main}(a) is not sufficient to conclude ballisticity in general, as per the following example.

\medskip

\begin{EXA}
\label{exa:speed0}
Choose $\mu(\gamma(e_1)=1)=p$ and for each $i\in \N$, 
\[\mu(\gamma(-e_1)=1-2^{-i}, \gamma(e_2)=\gamma(-e_2)=2^{-(i+1)}\big)=\frac{c}{i^2},\]
where $\sum_{i\in \N}c2^{-i}=1-p$. Then the expected time that the walk spends oscillating between $(0,0)$ and $(1,0)$ before moving to another site is infinite for all $p<1$.  For all $p$  sufficiently large Theorem \ref{thm:main}(a) holds, and the walker is transient in direction $e_1$, but the speed is zero for all $p<1$.   
\end{EXA}

\medskip

\begin{PRP}
\label{prp:E+}
Assume Conditions \ref{cond:orthogonal} and \ref{cond:ddim}.  

\noindent For each $d\ge 2$ there is a $p_d\in (1/2,1)$ such that if $\mu(\mathcal{G}_o\subset\mc{E}_+)> p_d$ then the condition of Theorem \ref{thm:main}(a) is satisfied.  
If $d=2$ then this holds with $p_d= p_c^{\smallOTSP}$.
\end{PRP}

\medskip

The condition in Proposition \ref{prp:E+} is a local condition saying that with high probability the only transitions allowed by the local environment lie in a cone pointing in direction $\ell_+$.  This condition is similar in spirit to the ``forbidden direction'' condition of \cite{RS06} and \cite{RS07}. They assume a direction that is forbidden with probability 1 (together with a so-called non-nestling assumption that makes ballisticity immediate) and then prove an invariance principle.  One might describe our assumptions on the environment as having a direction that is rare;y allowed rather than forbidden.  

The hypothesis of Proposition \ref{prp:E+} is not equivalent to Theorem \ref{thm:main}(a).  For example, the uniform $(\rightarrow \, \updownarrow)$ example (i.e.~$\mu(\gamma(e_1)=1)=p$ and $\mu(\gamma(e_2)=1/2=\gamma(-e_2))=1-p$) satisfies Theorem \ref{thm:main}(a) for every $p>0$.  Note that in this example $-e_1$ is a {\em forbidden direction}.  Similarly, if $\mu(\gamma(e_1)=1)=p_1$ and $\nu(\gamma(-e_1)=\hlf=\gamma(e_2))=p_2$ and $\mu(\gamma(e_2)=1/2=\gamma(-e_2))=1-p_1-p_2$ then Theorem \ref{thm:main}(a) will hold as long as $p_2$ is very small relative to $p_1$, even if $p_1$ itself is small.

Let $\mc{F}'_k=\sigma(X'_0,\dots, X'_k)$.  The following give verifiable conditions under which the condition of Theorem \ref{thm:main}(b) holds.

\medskip

\begin{PRP}
\label{prp:transverse} 
Assume Conditions \ref{cond:orthogonal} and \ref{cond:ddim}.
\begin{itemize}
\item[(a)] If for some $\ell'\ne o$ the transverse walk $X'_k$ is a submartingale (with bounded step size) under $P$ such that for some $\eta, \eta'>0$,
\[P(|X'_{k+1}- X'_{k}|>\eta|\mc{F}'_k)>\eta',\]
then the condition of Theorem \ref{thm:main}(b) is satisfied. 
\item[(b)] If $d\ge 2$ and $\mu$ is 2-valued then the condition of Theorem \ref{thm:main}(b) is satisfied.
\end{itemize}
\end{PRP}

\medskip

Proposition \ref{prp:transverse}(a) requires a projection of the walk to be a submartingale that can always move with probability bounded away from zero.  In the terminology of \cite{Zern98} this implies that the walk is either {\it non-nestling} or {\it marginal-nestling} in the transverse direction. 

The condition of Proposition \ref{prp:transverse}(a) does not hold for the following 2-valued example.

\medskip

\begin{EXA}
\label{exa:E_NSW}
For the uniform RWRE $(\rightarrow \NSW)$, the only projection that gives a submartingale is the projection in the direction $\pm e_2$.  However this martingale does not move at $\rightarrow$ sites, so it does not satisfy  Proposition \ref{prp:transverse}(a).
\end{EXA}

\medskip

Nevertheless, according to Proposition \ref{prp:transverse}(b), Theorem \ref{thm:main}(b) holds for Example \ref{exa:E_NSW}.  Therefore the walk of Example \ref{exa:E_NSW} will be ballistic in direction $\rightarrow$ as soon as $p>p_c^{\smallOTSP}$. We believe that in this example, our arguments prove ballisticity for a wider range of $p$, namely $p>p_c^{\smallFSOSP}$, where the latter is defined in \cite{HS_DRE1} (see the discussion preceding Theorem 3.13 of that paper). But we have not verified all the details.

As in Example \ref{exa:E_NSW}, the following is an immediate corollary of the above propositions and Theorem \ref{thm:main}.

\medskip

\begin{COR}
\label{cor:2valued} Let $d\ge 2$, and assume Conditions \ref{cond:orthogonal} and \ref{cond:ddim}.  

\noindent If $\mu$ is 2-valued and $\mu(\mc{G}_o\subset \mc{E}_+)>p_d$ then the model is ballistic in the direction $\ell_+$.
\end{COR}

\medskip

We suspect that one can replace $\mu$ being 2-valued with $\mu$ having {\em finite} support in Corollary \ref{cor:2valued}.  Note that Example \ref{exa:speed0} does not have finite support.

\medskip

The remainder of this paper is organised as follows.  In Section \ref{sec:ball} we recall some facts about directional transience, regeneration and ballisticity.  In Section \ref{sec:proof_main} we prove Theorem \ref{thm:main}.  In Section \ref{sec:propositionproof} we prove Proposition \ref{prp:transverse}.  In Section \ref{sec:percolation} we prove Proposition \ref{prp:E+} 
by examining percolation-type properties (the structure of forward clusters for certain degenerate random environments).  

Finally in Section \ref{sec:elliptic} we discuss other ballisticity conditions, and compare our results with those of the elliptic theory. In particular, we will find that in Example \ref{exa:NE_SW}, having strong barriers $\SW$ is an insurmountable obstacle to obtaining a positive speed in direction $\ell=(1,1)$ using one of the standard ballisticity conditions. One way of interpreting our results in the context of Example \ref{exa:NE_SW} is that we can overcome the presence of strong barriers $\SW$ by strengthening the forward push and including sufficiently many sites $\NE$ that don't permit backwards motion.

\section{Regeneration and Ballisticity}
\label{sec:ball}
In  non-elliptic environments (such as that of Example \ref{exa:NE_SW}) some sites may be unreachable by the walk.  Moreover, if Condition \ref{cond:orthogonal} does not hold then the range $\mc{R}$ of the random walk is finite.

Fix $\ell\in \R^d\setminus \{o\}$.  Let $A_+^\ell$ and $A_-^\ell$ denote the events that $X_n\cdot \ell\to\infty$ and $X_n\cdot \ell\to-\infty$ respectively.  The following is proved in \cite[Theorems 1.2--1.5]{RWDRE}, in most cases by adapting the methods of  Kalikow \cite{K81},  Sznitman and Zerner \cite{SZ99}, Zerner \cite{Zern02,Zern07} and Zerner and Merkl \cite{MZ01} to the non-elliptic setting.

\medskip

\begin{THM}[{\cite[ Theorems 1.2--1.5]{RWDRE}}]
\label{thm:old}
For i.i.d.~RWRE the following hold (for every $\ell\in \R^d\setminus \{o\}$):
\begin{itemize}
\item[(a)] $P(|\mc{R}|=\infty)\in \{0,1\}$, with $P(|\mc{R}|=\infty)=1$ if and only if Condition \ref{cond:orthogonal} holds.
\item[(b)]   $P(A_+^\ell\cup A_-^\ell)\in \{0,1\}$.
\item[(c)]  There exist deterministic $v_{+}(\ell), v_{-}(\ell)$ such that 
\[\lim_{n\ra \infty}\frac{X_n\cdot \ell}{n}=v_{+}(\ell)\mathbbm{1}_{A_+^\ell}+v_{-}(\ell)\mathbbm{1}_{A_{-}^\ell}, \quad P-\text{a.s.}\]
\item[(d)]  When $d=2$, $P(A_+^{\ell})\in \{0,1\}$ (hence a deterministic velocity vector $v$ always exists in 2-dimensions).
\item[(e)]  Assume that $\mu$ is 2-valued, and supported on $\{\gamma^{\sss(1)},\gamma^{\sss(2)}\}$, with $p=\mu(\gamma^{\sss(1)})$. Assume that the velocity $v=v[p]$ exists for each $p$. Then each coordinate of $v[p]$ is monotone in $p$.
\end{itemize}
\end{THM}

\medskip

Note that since $P(A)=E_{\nu}[\mP_{\omega}(A)]$ and $0\le \mP_{\omega}(A)\le 1$, $P(A)=1$ if and only if $\mP_{\omega}(A)=1$ for $\nu$-almost every $\omega$.  Similarly $P(A)=0$ if and only if $\mP_{\omega}(A)=0$ for $\nu$-almost every $\omega$.  

Theorem \ref{thm:old}(c) relies on a regeneration structure that is present on the event of directional transience.  This is well known in the uniformly elliptic setting, but perhaps less so in the non-elliptic setting.   For the purposes of this paper, fix $\ell\in \R^d\setminus o$ and assume almost sure transience in direction $\ell$, i.e.~that
\begin{equation}
\label{eqn:transience}
P(A^\ell_+)=1.
\end{equation}

For the remainder of this section, we assume Conditions \ref{cond:orthogonal} and \ref{cond:ddim}. The regeneration structure is then as follows (see the proof of Theorem 1.4 of \cite{RWDRE}, and note that the following makes a slight correction to how the structure was stated there).

Let $T_0=M_0=0$ and $D_0=\inf\{n>0:X_n\cdot \ell<0\}$. Let $T_1=\inf\{n:X_n\cdot \ell\ge 1\}$.  For $k\ge 1$ and $T_k<\infty$ let $D_k=\inf\{n>T_k:X_n\cdot \ell<X_{T_k}\cdot \ell\}$.  If $D_k<\infty$ then we let $M_k=\sup\{X_n\cdot \ell:n\le D_k\}$ and $T_{k+1}=\inf\{n>D_k:X_n\cdot \ell\ge M_k+1\}$.  Set $\Delta_{k+1}=M_{k+1}-M_k$.
Let $K=\inf\{k\ge 1:D_k<\infty\}$.   Then \eqref{eqn:transience} implies that $K<\infty$ a.s., and indeed, $P(D_k=\infty|T_k<\infty)$ is some fixed value $q>0$, so $K\ge 1$ is geometrically distributed;
\begin{equation*}
P(K> k)=(1-q)^k, \quad \text{ for $k\ge 1$.}
\end{equation*}
Thus we may define $\mathcal{T}_1=T_K$. This $\mathcal{T}_1$ acts as a regeneration time, as the process $\hat X_n=X_{\mathcal{T}_1+n}-X_{\mathcal{T}_1}$ and the environment $\hat \omega_x=\omega_{x+X_{\mathcal{T}_1}}$ (for $x\cdot \ell\ge 0$) are independent of the environment and walk observed up to time $\mathcal{T}_1$. This allows one to construct additional regeneration times $\mathcal{T}_{1}<\mathcal{T}_2<\dots$ such that the $X_{(\mathcal{T}_k+n)\land \mathcal{T}_{k+1}}-X_{\mathcal{T}_k}$ are i.i.d.~(over $k$) segments of path.  
Then the above discussion says that $\{(X_{\mc{T}_{k+1}}-X_{\mc{T}_{k}},\mc{T}_{k+1}-\mc{T}_k)\}_{k \in \N}$ are i.i.d.~copies of $(X_{\mc{T}_2}-X_{\mc{T}_1},\mc{T}_2-\mc{T}_1)$.  As a consequence (see e.g.~the proof of Theorem 1.4 of \cite{RWDRE}), 
$$
\label{speedformula}
v\cdot \ell=\frac{E[(X_{\mc{T}_2}-X_{\mc{T}_1})\cdot \ell]}{E[\mc{T}_2-\mc{T}_1]}=\frac{E[X_{\mc{T}_1}\cdot \ell\mid D_0=\infty]}{E[\mc{T}_1\mid D_0=\infty]}.
$$

Since for a unit vector $\ell$ we have that $\mc{T}_1\ge 1$ and $0<X_{\mc{T}_1}\cdot \ell<\mc{T}_1$ (by definition of $\mc{T}_i$) we immediately have the following well-known ballisticity criterion. 

\medskip

\begin{LEM}
\label{lem:ballisticitycriterion}
Assume \eqref{eqn:transience} as well as Conditions \ref{cond:orthogonal} and \ref{cond:ddim}. If $E[\mc{T}_1\mid D_0=\infty]<\infty$ then $v\cdot\ell >0$.  
\end{LEM}

\medskip

The corresponding criterion for an invariance principle is the following, which follows immediately using methods of \cite{Sz00}.

\medskip

\begin{LEM}
\label{lem:CLTcriterion}
Assume \eqref{eqn:transience} as well as Conditions \ref{cond:orthogonal} and \ref{cond:ddim}. Assume also that $E[\mc{T}^{2}_1\mid D_0=\infty]<\infty$. Then there exists a non-negative definite 
matrix $\Sigma$ (and $v\in \R^d\setminus \{o\}$) such that under the annealed/averaged measure $P$,
$$
\Big(\frac{X_{\lfloor nt\rfloor}-vnt}{\sqrt{n}}\Big)_{t\ge 0}\Rightarrow(Z_t)_{t\ge 0},\qquad \textrm{ as }n \ra \infty,
$$
where $Z_t$ is a $d$-dimensional Brownian motion with covariance matrix $\Sigma$,
and $\Rightarrow$ denotes weak convergence.
\end{LEM}

\medskip

Note that ellipticity enters the arguments of \cite{Sz00} in two ways: to obtain the regeneration structure, and to prove non-degeneracy of the covariance matrix. The former was extended to the non-elliptic case in \cite{RWDRE}, so those arguments carry over. The latter can actually fail in our setting (which is why we only claim non-negative definiteness of the covariance). For example, a RWRE in which $\mathcal{G}_x$ is either $\uparrow$ or $\rightarrow$ satisfies our hypothesis, yet its covariance is degenerate, because $X_n-\frac{n}2(e_1+e_2)$ is 1-dimensional. 

\section{Proof of Theorem \ref{thm:main}}
\label{sec:proof_main}
Fix $d,\alpha, \beta, \kappa, \ell,\ell'$ as in the theorem.  Without loss of generality we may assume that $\|\ell\|=\|\ell'\|=1$.  By hypothesis (a),  we may apply \cite[Theorem 2.7]{RWDRE} to conclude that $P(A_+^{\ell})=1$. Therefore the regeneration structure exists as described above.

Since $q=P(D_0=\infty)>0$ we can define $P_0(\cdot)$ to be the conditional probability measure $P(\cdot\mid D_0=\infty)$. 
We set $\mc{T}=\mc{T}_1=T_K$.  Note that $X_{\mc{T}}\cdot \ell\in (M_{K-1},M_{K-1}+2]$. 

To prove Theorem \ref{thm:main} note that by Lemmas \ref{lem:ballisticitycriterion} and \ref{lem:CLTcriterion} it suffices to find $C,\gamma,\delta>0$ such that 
\begin{equation}
\label{eqn:estimategoal}
P_0(\mc{T}>n)\le Ce^{-\gamma n^\delta}, \qquad \textrm{ for every }n.
\end{equation}
Choose $\alpha_1,\alpha_2,\alpha_3$ 
such that $0<\alpha_3<\alpha_2<\alpha_1+\alpha_2<\alpha$, and 
let $F_n$ be the event that $\mc{C}_o\subset -n^{\alpha_3}\ell+\mc{K}_{\kappa,\ell}$.  Then  
\begin{equation}
P_0(\mc{T}>n)\le P_0(F_n^c) +P_0(F_n, \mathcal{T}>n, M_{K-1}\le n^{\alpha})+P_0(F_n,  M_{K-1}>n^{\alpha}).\label{firstbreak}
\end{equation}
Note that $P_0(A)\le q^{-1}P(A)$ for any $A$, so hypothesis (a) of the Theorem shows that there exist $C_1,\gamma_1>0$ such that 
\begin{equation}
P_0(F_n^c)\le q^{-1}P(F_n^c)\le C_1e^{-\gamma_1 n^{\beta\alpha_3}}.\label{b1}
\end{equation}

To bound the second term on the RHS of \eqref{firstbreak}, note that the diameter of $\{x\in -n^{\alpha_3}\ell +\mc{K}_{\kappa,\ell}: x\cdot\ell\le n^\alpha\}$ is at most $C_3 n^\alpha$ for some $C_3$. Therefore if $F_n$ occurs, $\mathcal{T}>n$, and $M_{K-1}\le n^{\alpha}$ then $\max_{k\le n}X_k\cdot \ell\le n^{\alpha}$ and $R_n\le C_3n^\alpha$. Hypothesis (b) of the theorem then implies that there exist $C_2,\gamma_2>0$ such that 
\begin{align}
\nn P_0(F_n, \mathcal{T}>n, M_{K-1}\le n^{\alpha})&\le q^{-1}P(F_n, \mathcal{T}>n, M_{K-1}\le  n^{\alpha})\\
&\le C_2e^{-\gamma_2 n^\beta}.\label{b2}
\end{align} 

For the third term on the RHS of \eqref{firstbreak}, observe that if $K\le n^{\alpha_1}$ and $\Delta_k\le n^{\alpha_2}$ for each $k< K$ then $M_{K-1}=\sum_{k=0}^{K-1}\Delta_k\le n^{\alpha_1}\cdot n^{\alpha_2}<n^\alpha$. Therefore this term is bounded above by 
$$
P_0(K>n^{\alpha_1})+P_0(F_n, \, \exists k< K\le n^{\alpha_1} \textrm{ with }\Delta_k>n^{\alpha_2}).
$$
Since $K$ is geometrically distributed under $P$, the first term satisfies
\begin{align}
P_0(K>n^{\alpha_1})\le q^{-1}(1-q)^{n^{\alpha_1}}.\label{b3}  
\end{align}
It therefore remains to bound the quantity
$$
q^{-1}P(F_n, \, \exists k<K\le n^{\alpha_1} \textrm{ with }\Delta_k>n^{\alpha_2})
$$

Observe that if $\Delta_k>n^{\alpha_2}$ then there is a $j<D_{k}$ such that 
$$
X_{D_{k}}\cdot\ell+n^{\alpha_2}<  X_j\cdot\ell\le X_{D_{k}}\cdot\ell +n^{\alpha_2}+2.
$$
On the event 
$
\left\{\exists k<K\le n^{\alpha_1} \textrm{ with }\Delta_k>n^{\alpha_2}\right\}
$, 
if $k_1< n^{\alpha_1}$ is the first $k$ such that  $\Delta_k>n^{\alpha_2}$, and $j_1$ is the corresponding $j$ then $x\equiv X_{j_1}\in \mc{C}_o$ satisfies $\mc{C}_{x}\not\subset x-n^{\alpha_2}\ell + \mc{K}_{\kappa,\ell}$ and 
\begin{align*}
0\le x\cdot \ell\le &(k_1-1)n^{\alpha_2}+n^{\alpha_2}+2\le (k_1+2)n^{\alpha_2}\le 2n^{\alpha_1}n^{\alpha_2}
\le  2n^{\alpha}.
\end{align*}
On the event $F_n$, there are at most $C_4(n^{\alpha})^d$ points $x$ satisfying $x\in \mc{C}_o$ and $0\le x\cdot\ell \le 2n^{\alpha}$, which we collect in a (non-random) set $J$.
Therefore by hypothesis (a) and translation invariance,
\begin{align}
\nn P(F_n, \, \exists k< K\le n^{\alpha_1} \textrm{ with }\Delta_k>n^{\alpha_2})&\le P\Big(\bigcup_{x\in J}\{\mc{C}_{x}\not\subset x-n^{\alpha_2}\ell + \mc{K}_{\kappa,\ell}\}\Big)\\
&\le \sum_{x\in J}P\left(\mc{C}_{x}\not\subset x-n^{\alpha_2}\ell + \mc{K}_{\kappa,\ell}\right)\nn\\
&\le C_1C_4n^{d\alpha}e^{-\gamma_1 n^{\beta\alpha_2}}.\label{b4}
\end{align}
Every one of the bounds \eqref{b1},\eqref{b2},\eqref{b3},\eqref{b4} can be rewritten as $Ce^{-\gamma n^\delta}$ for some single choice of $\delta$ and $\gamma$, so combining them establishes the desired bound \eqref{eqn:estimategoal}.\Qed

\section{Proof of Proposition \ref{prp:transverse} }
\label{sec:propositionproof}

The following Lemma, when applied to the transverse walk $X'$ proves Proposition \ref{prp:transverse}(a).

\medskip

\begin{LEM}
\label{lem:martrange} Let $M_k$ be a submartingale with respect to a filtration $\mc{F}_k$. Assume that $M_0=0$, $\delta\le E[(M_{k+1}-M_k)^{2}\mid\mc{F}_k]$ for some $\delta>0$, and $|M_{k+1}-M_k|\le m$ for some $m<\infty$. Then there exist $C,m_0,\gamma>0$, depending only on $\delta$ and $m$, such that
$$
P\big(\max_{k\le n}|M_k|\le y\big)\le Ce^{-\gamma n/y^{2}}
$$ 
for every $n\ge 1$ and $y\ge m_0$.
\label{lem:marts}
\end{LEM}

\medskip

To see that this implies Proposition \ref{prp:transverse}(a), note that as the walk $X$ is a nearest neighbour walk, we may apply Lemma \ref{lem:marts} to the submartingale $X_k'$ to see that there exist $C',\gamma'>0$ such that for every $n \in \N$,
\begin{align}
P\big(\max_{k\le n}|X'_k|\le y\big)\le C'e^{-\gamma' n/y^{2}}.\label{marty}
\end{align}
Letting $0<\alpha<\frac12$ and $y=Cn^{\alpha}$, then \eqref{marty} implies that for every $n \in \N$
\begin{align}
P\big(\max_{k\le n}|X'_k|\le Cn^{\alpha}\big)\le C'e^{-\frac{\gamma'}{C^{2}}n^{1-2\alpha}},\label{rangebound}
\end{align}
which establishes Proposition \ref{prp:transverse}(a) with $\beta\in (0,1-2\alpha]$, $C_2=C'$ and $\gamma_2=\frac{\gamma'}{C^{2}}$.

\bigskip

\noindent {\em Proof of Lemma \ref{lem:martrange}}.
Our proof is motivated by the quasi-stationary distribution for Brownian motion on an interval. Consider $g(u)=\cos(\frac{\pi}{4}+u)$. Then $g''+g=0$, $g\le 1$, and $|g'''|\le 1$. Fix $\frac12<a<\frac{\pi}{4}$. We can choose $\epsilon>0$ so that for $u\in[-a,a]$ we have $g'(u)\le 0$ and $0<\epsilon\le g(u)$. Let $\Delta=M_{k+1}-M_k$ and $\gamma=\frac{\delta}{16}$. By Taylor's theorem, provided that $|\frac{M_k}{2y}|\le a$ we have

\begin{align}
\nonumber
E\left[g\Big(\frac{M_{k+1}}{2y}\Big)\Big|\mc{F}_k\right]
\le &E\left[g\Big(\frac{M_{k}}{2y}\Big)+\frac{\Delta}{2y}g'\Big(\frac{M_{k}}{2y}\Big)
+\frac{\Delta^2}{8y^2}g''\Big(\frac{M_{k}}{2y}\Big)+\frac{m^3}{48y^3}\Big|\mc{F}_k\right]\\ 
=&g\Big(\frac{M_{k}}{2y}\Big)-g\Big(\frac{M_{k}}{2y}\Big)E\left[\frac{\Delta^2}{8y^2} | \mc{F}_k\right]+\frac{m^3}{48y^3}\nn\\
&+g'\Big(\frac{M_{k}}{2y}\Big)E\left[\frac{\Delta}{2y}|\mc{F}_k\right].\label{jolly1}
\end{align}
Since $M$ is a submartingale we have $\E[\Delta|\mc{F}_k]>0$.  Now using the facts that $y>0$, $|\frac{M_k}{2y}|\le a$ and $g'<0$ on $[-a,a]$ we can bound the final term \eqref{jolly1} above by 0 to get
\begin{align}
E\left[g\Big(\frac{M_{k+1}}{2y}\Big)\Big|\mc{F}_k\right]&\le g\Big(\frac{M_{k}}{2y}\Big)\left(1-\frac{E\left[\Delta^2 | \mc{F}_k\right]}{8y^2}\right)+\frac{m^3}{48\epsilon y^3}\epsilon \nonumber \\
&\le\Big(1-\frac{\delta}{8y^2}+\frac{m^3}{48\epsilon y^3}\Big)g\Big(\frac{M_{k}}{2y}\Big),\label{martlast}
\end{align}
where we have used the fact that $g(M_k/(2y))\ge \epsilon$ when $|\frac{M_k}{2y}|\le a$.

Choose $y_0>0$ such that $\frac{m^3}{48\epsilon y_0}<\frac{\delta}{16}$ and $\frac12+\frac{m}{2y_0}\le a$. Let $y\ge y_0$. Then \eqref{martlast} is bounded above by 
$(1-\frac{\delta}{16y^2})g(\frac{M_k}{2y})\le e^{-\gamma/y^2}g(\frac{M_k}{2y})$, so $e^{\gamma k/y^2}g(\frac{M_k}{2y})$ is a supermartingale (while $|\frac{M_k}{2y}|\le a$). Let $T=\inf\{n: |M_n|> y\}$.  Since $|M_{(T-1)\wedge n}|\le y$ it follows that $|M_{T\wedge n}|\le y+m$, so $|\frac{M_{T\wedge n}}{2y}|\le \frac12+\frac{m}{2y}\le a$ by choice of $y$.  

Now observe that
\begin{align*}
P\big(\max_{k\le n}|M_k|\le y\big)\le P(T>n)&\le e^{-\gamma n/y^2}\min\big(e^{\gamma n/y^2},\E[e^{\gamma T/y^2}]\big),\\
&= e^{-\gamma n/y^2}\E\big[e^{(\gamma \wedge T)n/y^2}\big]\\
&\le \epsilon^{-1} e^{-\gamma n/y^2}E\left[e^{\gamma(T\land n)/y^2}g\Big(\frac{M_{T\land n}}{2y}\Big)\right]\\
&\le \epsilon^{-1} e^{-\gamma n/y^2}E\left[g(0)\right]=\frac{1}{\sqrt{2}\epsilon} e^{-\gamma n/y^2},
\end{align*}
where we have used optional sampling to obtain the last inequality.
\qed

\bigskip

\noindent
{\bf Remark:} For a related estimate see \cite[Proposition 4.1]{MPRV}. 

\medskip


\bigskip

For 2-valued models in which the local environments are $\gamma^{\sss(1)}$ or $\gamma^{\sss(2)}$, it is useful to consider the local biases $\vec{u}^{\sss(i)}=\sum_{i=1}^d (\gamma^{\sss(i)}(e_i)-\gamma^{\sss(i)}(-e_i))e_i$.
Consider, for example, a 2-valued model in 2 dimensions with $\mc{S}(\gamma^{\sss(1)})=e_1$ and $\mc{S}(\gamma^{\sss(2)})=\{-e_1,e_2,-e_2\}$ (i.e.~the induced random graph is the same as in Example \ref{exa:E_NSW}).  If $\vec{u}_2^{\sss(2)}\ne 0$ then we can find a direction $\ell$ in which both environments induce a  drift.  If (as in Example \ref{exa:E_NSW}) $\vec{u}_2^{\sss(2)}= 0$ then  $X_n\cdot e_2$ is a martingale that does not move when $X_n$ is at a $\gamma^{\sss(1)}$ environment.  This martingale therefore does not satisfy the conditions of Lemma \ref{lem:martrange}.  Nevertheless Corollary \ref{cor:2valued} shows that it is still ballistic.

The following result gives various cases in which Lemma \ref{lem:martrange} applies directly to 2-valued models. 

\medskip

\begin{LEM}
\label{lem:2valuedrange}
Let $\mu(\gamma^{\sss(1)})=p=1-\mu(\gamma^{\sss(2)})\in (0,1)$ be a 2-valued model satisfying Conditions \ref{cond:orthogonal} (for a set $V$) and \ref{cond:ddim} with $d\ge 2$.   
\begin{itemize}
\item[(I)] If $\vec{u}^{{\sss(1)}}\ne o$ and  $\vec{u}^{\sss(2)}\ne o$ and $\vec{u}^{\sss(2)}\ne -c\vec{u}^{\sss(1)}$ for any $c>0$ then there exists $\ell'\ne o$ such that $\vec{u}^i\cdot \ell'>0$ for $i=1,2$. For this $\ell'$, Lemma \ref{lem:martrange} applies to the transverse walk $X'_k$.
\item[(II)] If $\vec{u}^{\sss(2)}=-c\vec{u}^{\sss(1)}$ for some $c>0$, and if $\vec{u}^{\sss(1)}\perp \ell'$ for some $\ell'= \sum_{e\in V}x_e e$ with all $x_e\neq 0$, then Lemma \ref{lem:martrange} applies to $X'_k$.
\item[(III)] If $\vec{u}^{\sss(2)}=o$ then Lemma \ref{lem:martrange} applies to $X'_k$, for one of $\ell'=\pm\sum_{e\in V}e$. 
\end{itemize}
\end{LEM}
\proof (I) If $\vec{u}^{{\sss(1)}}\ne o$ and  $\vec{u}^{\sss(2)}\ne o$ then the $\{\ell:\vec{u}^{\sss(i)}\cdot \ell>0\}$ are half spaces, which must intersect unless one is the negative of the other. The latter possibility is ruled out, since it would imply that
$\vec{u}^{\sss(2)}= -c\vec{u}^{\sss(1)}$ for some $c>0$. Therefore we can find an $\ell'$ in the intersection. It follows that $X'_k$ is a submartingale. Because $\vec{u}^{\sss(i)}\cdot \ell'\neq 0$ for each $i$, there is a positive probability of movement in either environment. 





(II) $\vec{u}^{\sss(2)}\cdot \ell'=0=\vec{u}^{\sss(1)}\cdot \ell'$ so $X'_k$ is a martingale.  By Condition \ref{cond:orthogonal}, it is possible to take a step of size at least $\min_{e\in V}|x_e|$ in either environment. 

(III) Either $\ell'=\sum_{e\in V}e$ or $\ell'=-\sum_{e\in V}e$ will have $\vec{u}^{\sss(1)}\cdot \ell'\ge 0$, so that $X'_k$ is a submartingale. Again, there is a positive probability of movement in either environment by Condition \ref{cond:orthogonal}.
\qed

\blank{
\begin{LEM}
\label{lem:livesincone}
Let $\mu(\gamma^{\sss(1)})=p=1-\mu(\gamma^{\sss(2)})\in (0,1)$ be a 2-valued model. If hypothesis (a) of Theorem \ref{thm:main} holds, then at least one of $\gamma^{\sss(1)}$ or $\gamma^{\sss(2)}$ is supported on an orthogonal set.
\end{LEM}
\proof If neither $\gamma^{\sss(1)}$ nor $\gamma^{\sss(2)}$ is supported on an orthogonal set, there are $i,j$ such that $\pm e_i\in\mc{S}(\gamma^{\sss(1)})$ and $\pm e_j\in\mc{S}(\gamma^{\sss(2)})$. If $i=j$ then $\mathbb{Z}e_i\subset\mc{C}_0$ which contradicts hypothesis (a). 

So assume $i\neq j$. By Theorem 4.9 of \cite{HS_DRE1}, the intersection of $\mc{C}_o$ and the hyperplane $H$ spanned by $e_i$ and $e_j$ is infinite in every direction of $H$. This contradicts hypothesis (a).
\qed
\medskip
}

\medskip

For any 2-valued model $\mu(\gamma^{\sss(1)})=p=1-\mu(\gamma^{\sss(2)})\in (0,1)$, let $N^{\sss(i)}_n=\#\{0\le m<n:\omega_{X_m}=\gamma^{\sss(i)}\}$ and note that $N^{\sss(1)}_n+N^{\sss(2)}_n=n$.  Let us write $N_n$ for $N_n^{\sss(1)}$.

\medskip

\noindent {\em Proof of Proposition \ref{prp:transverse}(b)}.  If either $\vec{u}^{{\sss(1)}}$ or $\vec{u}^{{\sss(2)}}$ equals $o$, or if $\vec{u}^{\sss(2)}\neq -c\vec{u}^{\sss(1)}$ for any $c>0$, then the claim holds by Lemmas \ref{lem:martrange} and \ref{lem:2valuedrange}. 

So assume that $\vec{u}^{\sss(2)}= -c\vec{u}^{\sss(1)}\neq o$, for some $c>0$.  Write $\vec{u}=(u_1,\dots, u_d)$ for $\vec{u}^{\sss(1)}$.  Without loss of generality there exists $k\le d$ such that $e_i \cdot \vec{u}> 0$ for $i=1,\dots, k$, and $e_i \cdot \vec{u}=0$ for each $i=k+1,\dots, d$.

If $k>1$ then the vector $\ell'=(-(u_2+\dots+u_d),u_1, u_1, \dots, u_1)\perp \vec{u}$, and by Lemma \ref{lem:2valuedrange} the transverse walk $X'$ for this $\ell'$ is a martingale to which Lemma \ref{lem:martrange} applies.

Therefore we will assume for the rest of the proof that $k=1$, so $\vec{u}^{\sss(1)}=u_1e_1$ and $\vec{u}^{\sss(2)}=-cu_1e_1$.  For each $i\ge 2$, condition \ref{cond:ddim} implies that either $\gamma^{\sss(1)}(e_i)=\gamma^{\sss(1)}(-e_i)>0$ or $\gamma^{\sss(2)}(e_i)=\gamma^{\sss(2)}(-e_i)>0$ (or both).  
If $\gamma^{\sss(1)}(e_{i_1})=\gamma^{\sss(1)}(-e_{i_1})>0$ and $\gamma^{\sss(2)}(e_{i_2})=\gamma^{\sss(2)}(-e_{i_2})>0$ for some $i_1, i_2\ge 2$ (possibly equal) ,then let $\ell'=e_{i_1}+e_{i_2}$. We see that the transverse walk $X'$ for this $\ell'$ is a martingale, and Lemma \ref{lem:martrange} applies since $X'$ has a positive probability of moving in either environment.

It remains to handle the case that one of the $\gamma^{\sss(i)}$ (which we will take to be $\gamma^{\sss(1)}$) is supported on $\pm e_1$.  In this case by Condition \ref{cond:ddim}, $\gamma^{\sss(2)}(e_i)=\gamma^{\sss(2)}(-e_i)>0$ for each $i\ge 2$ (i.e.~we have basically reduced the problem to something like Example \ref{exa:E_NSW}).

Let $\delta>0$, and let $B_n=\{n^{-1}N_n\le 1-\delta\}$.  Then on $B_n$ we have at least $\delta n$ departures from $\gamma^{\sss(2)}$ sites by time $n$. Let $\ell'=(0,1,1,\dots,1)$ and let $X'_k=X_k\cdot \ell'$ be the transverse walk. Let $\tilde X_n$ be $X'_k$ time changed by $N_n^{\sss(2)}$, i.e.~so that time only advances at $\gamma^{\sss(2)}$ sites. It is a martingale and Lemma \ref{lem:martrange} applies to $\tilde X_n$, so 
there is a $C'$ with
$$
P\big(\max_{k\le n}|X'_k|\le y, B_n\big)\le P\big(\max_{k\le \delta n}|\tilde X_k|\le y\big)\le C'e^{-\gamma \delta n/y^{2}}.
$$ 
As in \eqref{rangebound} this implies that 
\begin{equation}
P(R_n< Cn^\alpha, B_n)\le C'e^{-\frac{\gamma\delta}{C^2}n^{1-2\alpha}}.
\label{firsthalf}
\end{equation}
We therefore take $\alpha<\frac12$ and $\beta=1-2\alpha$. 

Now let $\Xi_k=\pm 1$ according to whether the $k$th departure from a $\gamma^{\sss(1)}$ site is $\pm e_1$. In other words, the $\Xi_k$ are independent, with  $P(\Xi_k=1)=\gamma^{\sss(1)}(e_1)$ and $P(\Xi_k=-1)=\gamma^{\sss(1)}(-e_1)$, so $E[\Xi_k]=u_1$.  Let $Y_n=\sum_{k=1}^n \Xi_i$.  Choose $\delta<\frac{u_1}{4}$. On $B_n^c$ there are then at most $\frac{u_1n}{4}$ departures from $\gamma^{\sss(2)}$ sites by time $n$. So if $Y_{N_n}>\frac{u_1n}{2}$, it follows that $X_1\cdot e_1>\frac{u_1n}{4}$ and hence $R_n>\frac{u_1n}{4}$. Since $\alpha<\frac12$ for each 
 $C$ we have 
$Cn^\alpha<\frac{u_1n}{4}$ for all $n\ge n_C$.  Therefore, for such $n$,
$$
P(R_n< Cn^\alpha, B_n^c)
\le P(Y_{N_n} <\frac{u_1 n}{2}, B_n^c)\le \sum_{k=(1-\delta)n}^n P(Y_k< \frac{u_1 n}{2}).
$$
By Cramer's theorem (see e.g.~\cite{DeZ}), there exist $c$, $c'$ such that this is $\le nc'e^{-cn}$. Since $\beta<1$ we may combine this estimate with \eqref{firsthalf} to obtain the bound of Theorem \ref{thm:main}(b), for large $n$.  Raising these constants we obtain the bound for all $n$.
\qed

\section{Proof of Proposition \ref{prp:E+}}
\label{sec:percolation}
Consider a model with 2-valued support of the form $\mu(\mc{E}_+)=p$ (i.e.~$\mu(\gamma:\mc{S}(\gamma)=\mc{E}_+)=p$) and $\mu(\mc{E})=1-p$.  Define $\ell_+=\sum_{e\in \mc{E}_+}e=(1,\dots,1)$ and 
\begin{align}
p_c(d)=\inf\Big\{p>0:&\exists \kappa>0 \text{ such that }\nn\\
&\nu\big(\cup_{n=1}^{\infty}\{\mc{C}_o(p)\subset-n\ell_++\mc{K}_{\kappa, \ell_+}\}\big)=1\Big\}.\nn
\end{align} 
In other words, for $p>p_c(d)$ the forward cluster of such a model is contained in a cone.  It is an immediate consequence of \cite[Theorem 1.6]{HS_DRE2} that $p_c(d)>.5730$ for all $d\ge 2$.  Since $\mc{E}_+\subset  \mc{E}$, under the natural coupling of environments for all $p$ (i.e.~$\mc{G}_x=\mc{E}_+$ if and only if $U_x\le p$, where $U_x\sim U[0,1]$ are independent) $\mc{C}_o(p)$ is monotone decreasing in $p>p_c(d)$.  We conclude the following. 

\medskip

\begin{LEM}
\label{lem:pcd}
Suppose that $\mu(\mc{E}_+)=p$, $\mu(\mc{E})=1-p$.  Then for all $p>p_c(d)$ there exist $\kappa=\kappa(p,d)>0$ such that 
\begin{align}
\nu\big(\cup_{n=1}^{\infty}\{\mc{C}_o(p)\subset -n\ell_++\mc{K}_{\kappa,\ell_+}\}\big)=1,\label{yoyo1}
\end{align}
moreover $\kappa(p,d)$ is non-decreasing in $p$ for each $d$.  For $p<p_c(d)$, \eqref{yoyo1} fails for every $\kappa>0$.
\end{LEM}

\medskip

Although we believe that Theorem \ref{thm:main}(a) does hold in this setting as soon as $p>p_c(d)$, Lemma \ref{lem:pcd} is not sufficient to establish that result as it does not give tail probabilities.  Let $\sigma_d$ be the connective constant for self-avoiding walks on the cubic lattice $\mathbb{Z}^d$, defined as $\lim_{N\to\infty}c_N^{1/N}$, where $c_N$ is the number of self-avoiding walks of length $N$. Let $p_d=1-\sigma_d^{-2}$.  The following result (based on \cite[Theorem 4.2]{HS_DRE1} and proved below) verifies that $p_c(d)\le p_d<1$ for each $d$, and gives bounds for the relevant tail probabilities when $p>p_d$.

\medskip

\begin{LEM}
\label{lem:SAW} 
Let $d\ge 2$. Consider an i.i.d. RWRE in which $\mu(\gamma:\mc{S}(\gamma)\subset\mc{E}_+)>p_d$.
Then  $\exists$  constants $C,\kappa,\gamma>0$ such that $\nu(\mc{C}_o\subset -n\ell_++\mc{K}_{\kappa,\ell_+})\ge 1-Ce^{-\gamma n}$ for every $n$. 
\end{LEM}

\medskip

With additional conditions on $\mu$, the constant $p_d$ in the above result may be improved slightly (see \cite{HS_DRE1}, as well as for a table of values for $\sigma_d$). When $d=2$, duality with an oriented percolation model (whose critical percolation probability is $p_c^{\smallOTSP}\approx .5956$) allow us to prove the following.

\medskip

\begin{LEM}
\label{lem:basiclemma}
Fix $d=2$.  Then $p_c(2)=p_c^{\smallOTSP}$. Moreover,  if $\mu(\gamma:\mc{S}(\gamma)\subset\mc{E}_+)>p_c(2)$ there exist constants $C,\kappa,\gamma>0$ such that 
\begin{align*}
\nu(\mathcal{C}_o\subset -n\ell_++\mc{K}_{\kappa,\ell_+})\ge 1-Ce^{-\gamma n}.
\end{align*}
\end{LEM}

\medskip

Clearly Lemmas \ref{lem:SAW} and \ref{lem:basiclemma} imply Proposition \ref{prp:E+}.   Therefore to prove the proposition it is sufficient to prove each of the lemmas.

The proof of Lemma \ref{lem:SAW} is a relatively straightforward adaptation of the proof of \cite[(4.1)]{HS_DRE1}.

\medskip

\noindent {\em Proof of Lemma \ref{lem:SAW}}.
Set $p=\mu(\mathcal{G}_o\subset\mc{E}_+)$. Assume that $p>1-\sigma_d^{-2}$, in other words, that $\sqrt{1-p}<\frac{1}{\sigma_d}$. We may therefore find a $\theta<\frac12$ and a $\mu>\sigma_d$ such that $(1-p)^\theta <\frac{1}{\mu}$. Now choose $\kappa<1-2\theta$. If $c_N$ denotes the total number of $N$-step self avoiding walks from $o$ then we may, by definition of $\sigma_d$, find a constant $C$ such that $c_N\le C\mu^N$ for every $N$. 

Set $\Gamma=(-n\ell_+ + \mc{K}_{\kappa,\ell_+})^c$. Fix, for the moment, a lattice point $y\in\Gamma$ and self-avoiding path $(w_0,\dots,w_N)$ from $o=w_0$ to $y=w_N$. Clearly $N\ge n$ (since $\kappa<1$) 
Suppose that at most a fraction $\theta$ of the steps of the path are from $\mc{E}_-$. Then $y\cdot \ell_+\ge N(1-\theta) -N\theta=N(1-2\theta)>\kappa N$. We also have $\kappa<1<\sqrt{d}$, so 
$$
(y+n\ell_+)\cdot\ell_+
\ge \kappa N +nd
\ge \kappa\|y\|+\kappa n\sqrt{d}
=\kappa\|y\|+\kappa\|n\ell_+\|
\ge\kappa\|y+n\ell_+\|.
$$
In other words, $y\in -n\ell_+ +\mc{K}_{\kappa,\ell_+}=\Gamma^c$, which is impossible. Therefore, at least $N\theta$ of the steps belong to $\mc{E}_-$, so the probability that this particular path is actually a $\mc{G}$-path is at most $(1-p)^{N\theta}$. 

If $\mc{C}_o$ intersects $\Gamma$ then there is a self-avoiding $\mc{G}$-path from $o$ to some point in $\Gamma$. By the above estimate, 
$$
\nu(\text{$\mc{C}_o$ intersects $\Gamma$})
\le \sum_{N=n}^\infty c_N(1-p)^{N\theta}
\le \sum_{N=n}^\infty C\Big(\mu(1-p)^\theta\Big)^N
=C'e^{-\gamma n}
$$
where $e^{-\gamma}=\mu(1-p)^\theta$.
\Qed

For the comparable result (in dimension $d=2$) our arguments rely on estimates for oriented percolation as in \cite{Dur84} (see Lemmas \ref{lem:Durrett} and \ref{lem:abovealine} below).  Recall that $p_c^{\smallOTSP}$ denotes the critical percolation parameter for oriented site percolation on the triangular lattice.
It is shown in \cite[Prop.~3.1]{HS_DRE2} that $\mathcal{C}_o$ has a lower boundary if and only if $p>p_c^{\smallOTSP}$.  Moreover by \cite[Theorem 1.5(III)]{HS_DRE2},
if $p>p_c^{\smallOTSP}$ then this boundary almost surely has an asymptotic slope of $\rho_p<-1$ in the northwest direction and $1/\rho_p>-1$ in the southeast direction, so $\#\{x\in \mc{C}_o:x\cdot \ell_+<0\}$ is almost surely finite.  Label any vertex $y$ as {\it open} if $\mathcal{G}_y\subset\NE$. An {\it open path} is a sequence of open vertices $y_i$ such that each $y_{i+1}-y_i\in \OTSP=\{-e_1,e_2,e_2-e_1\}$.  The idea is that 
an infinite oriented open path in both directions in the triangular lattice (generated by 
$(\leftrightarrow, \updownarrow, \Nwsearrow)$ lines) that passes below $o$, also passes below $\mc{C}_o$.

\medskip

\noindent{\em Proof of Lemma \ref{lem:basiclemma}.}
Let $p=\mu(\mathcal{G}_o\subset\NE)>p_c^{\smallOTSP}$ and choose $\theta$ such that $\rho_p<\theta<-1$. Let $L$ and $L'$ be the lines with slope $\theta$ and $1/\theta$ through $(-1,-1)$. Let $\Gamma$ denote the set of $x$ lying above both $L$ and $L'$. 


Choose $\epsilon$ so that $0<\epsilon<\frac{1}{\theta}-\theta$. Let $A$ and $A'$ be the line segments $\{(0,z): z\in[\frac{1}{\theta}-1-\epsilon,\frac{1}{\theta}-1]\}$ and $\{(z,0): z\in[\frac{1}{\theta}-1-\epsilon,\frac{1}{\theta}-1]\}$. Therefore $A$ lies below $L'$ and above $L$, while $A'$ lies below $L$ and above $L'$. 

By Lemma \ref{lem:abovealine} below we may discard an event of probability at most $Ce^{-\gamma n}$ and obtain an infinite open path from some site in $nA$ that lies above $nL$. By symmetry, we may discard a further event of probability at most $Ce^{-\gamma n}$ and obtain an infinite open path terminating at some site in $nA'$ that lies above $nL'$. By construction, these paths must cross somewhere in $(-\infty,0]\times(-\infty,0]$, so following first one and then the other gives us an open path that is infinite in both directions. It lies above both $nL$ and $nL'$, and separates these lines from $o$. As remarked above, this implies that $\mathcal{C}_o\subset n\Gamma=-n\ell_++\mc{K}_{\kappa,\ell_+}$ for a suitable choice of $\kappa$.
\Qed

In the remainder of this section, we specialize to $d=2$, and will consider various estimates for oriented site percolation on the triangular lattice $\mathbb{Z}^{\sss(2)}$. We realize the latter using the vertices of $\mathbb{Z}^{2}$ connected by horizontal and vertical bonds, as well as by bonds of slope $-1$. Given $p$, let sites in $\mathbb{Z}^{\sss(2)}$ be open with probability $p$, independently of each other. 
As above, we call a sequence $(\dots, y_{-1},y_0, y_1, \dots)$ -- finite or infinite  -- an \emph{open path} if each $y_i$ is an open site, and each $y_{i+1}-y_i\in \OTSP=\{-e_1,e_2,e_2-e_1\}$.
For any site $x\in \mathbb{Z}^{\sss(2)}$, let its forward cluster $\mathbf{C}_x$ be the set of sites $y\in \mathbb{Z}^{\sss(2)}$ for which there is an open path starting at $x$ and ending at $y$.  Let $\mathbf{C}_x^\infty$ be the set of $y\in\mathbf{C}_x$ such that $|\mathbf{C}_y|=\infty$. For $A\subset \mathbb{Z}^{\sss(2)}$, set $\mathbf{C}_A=\cup_{x\in A}\mathbf{C}_x$ and $\mathbf{C}_A^\infty=\cup_{x\in A}\mathbf{C}_x^\infty$. In other words, each site in $\mathbf{C}_A^\infty$ can be reached from $A$ by an open path, and is then left via an infinite open path.

For $Y=\{(0,z)\in\mathbb{Z}^{\sss(2)}: z\le 0\}$ set $\bar{u}_n = \max\{y: (-n,y)\in\mathbf{C}_Y\}$. 
Also let $\tau_o=\sup\{y-x: (x,y)\in \mathbf{C}_o\}$, which measures the furthest diagonal line reached by the forward cluster of the origin.  
More generally, if $A\subset \mathbb{Z}^{\sss(2)}$, let $\tau_A=\sup\{y-x: (x,y)\in \mathbf{C}_A\}$.
Note that for $A$ finite, $|\mathbf{C}_A|=\infty\Leftrightarrow \tau_A=\infty$. If we wish to measure diagonal displacement relative to a point $z=(x_0,y_0)$ other than $o$, we will use $\tau_A^z=\tau_A-(y_0-x_0)$. 

Let $p_c^{\smallOTSP}$ denote the critical probability for oriented site percolation on the triangular lattice. The following bounds are known: 
$$
0.5699\le p_c^{\smallOTSP}\le 0.7491;
$$
the former is Corollary 6.3 of of \cite{HS_DRE2}, while the latter follows from the square lattice bound $\le p_c^{\smallNE}\le 0.7491$ of Balister et al \cite{BBS}, since $p_c^{\smallOTSP}\le p_c^{\smallNE}$. 

Fix $p>p_c^{\smallOTSP}$. Proposition 4.1 of \cite{HS_DRE2} (reformulating results in \cite{Dur84}) states that  there exists a $\rho_p<-1$ such that on the event $\{|\mathbf{C}_o|=\infty\}$, the set $\mathbf{C}_o$ has an upper boundary with asymptotic slope $\rho_p$ and a lower boundary with asymptotic slope $1/\rho_p$.  

\medskip

\begin{LEM}
\label{lem:Durrett}
Let $p>p_c^{\smallOTSP}$ and choose $\theta_1$ with $\rho_p<\theta_1<-1$. Then
\begin{enumerate}
\item $\exists$ a constant $\gamma_1>0$ such that $P(\bar u_n\le -n\theta_1)\le e^{-\gamma_1 n}$, $\forall n$;
\item $\exists$ constants $C_2$, $\gamma_2>0$ such that $P(n\le \tau_o<\infty)\le C_2e^{-\gamma_2 n}$, $\forall n$;
\item $\exists$ a constant $\gamma_3>0$ such that $P(\tau_A<\infty)\le e^{-\gamma_3 |A|}$, $\forall A\subset Y$.
\end{enumerate}
\end{LEM}
\begin{proof}
These are all taken from \cite{Dur84}.  The lattice used there is different from ours, but it can be verified that the arguments all apply equally well in our setting. 
See also Section 4 of \cite{HS_DRE2} where a similar translation is carried out. In particular, (a) is formula (11.1) of \cite{Dur84}, (b) is formula (12.1), and (c) is formula (10.5). 
\end{proof}

We will follow the convention that constants $C$ and $\gamma$ may change from line to line. If specific values are to be tracked, we will index them (as in the above result).

For $\theta_1$ as above, choose 
$\epsilon>0$, and $\theta$ with 
$\theta_1<\theta<-1$. Set $A_{n,\epsilon}=\{(n,y)\in\mathbb{Z}^{\sss(2)}: 0\le y\le\epsilon n\}$. Let $L$ be the line through $o$ with slope $\theta$, and let $L_n$ be the line through $(n,0)$ with slope $\theta_1$. We are interested in the event
$$
\mathcal{A}_{n,\epsilon,\theta}=\{\text{$\exists$ infinite open path, starting in $A_{n,\epsilon}$ and lying above $L$}\}.
$$

\medskip

\begin{LEM}
\label{lem:abovealine}
Let $p>p_c^{\smallOTSP}$. Choose $\theta_1$ and $\theta$ with $\rho_p<\theta_1<\theta<-1$, and choose $\epsilon>0$. There are constants $C>0$ and $\gamma>0$ such that $P(\mathcal{A}_{n,\epsilon,\theta})\ge 1-Ce^{-\gamma n}$ for every $n$. 
\end{LEM}
\begin{proof}
We will temporarily fix $k\ge 0$, and will estimate the probability that there exists a point $(-k,z)\in \mathbf{C}_{A_{n,\epsilon}}^\infty$ that lies above $L$. Constants below are as taken from Lemma \ref{lem:Durrett}

Discarding an event of probability at most  $e^{-\gamma_3\epsilon n}$, there is an infinite open path $\sigma_1$, starting from some point $x_1$ of $A_{n,\epsilon}$. Let $y_1=(-k,z_1)$ be the first site on $\sigma_1$ whose first coordinate equals $-k$. Discarding a further event, of probability at most $e^{-\gamma_1(k+n)}$ there is a also an open path $\sigma_2$ from some site $x_2=(n,z_2)$ with $z_2\le 0$, to a point $y_2=(-k, z_2')$ lying above $L_n$. Let $k_3$ be the first integer exceeding $-\theta k$, and let $x_3=(z_3,k_3)$ be the first site on $\sigma_2$ whose second coordinate exceeds $-\theta k$. If $\mathbf{C}_{x_3}$ is infinite, we claim that there will be a point $y=(-k,z)\in \mathbf{C}_{A_{n,\epsilon}}^\infty$. 

To see this, we know there is an infinite open path $\sigma_3$ starting at $x_3$. Let $y_3$ be the first site on $\sigma_3$ whose first coordinate equals $-k$. By construction, $y_3$ lies above $L$. If $y_1$ lies above $L$ then take $y=y_1\in \mathbf{C}_{A_{n,\epsilon}}^\infty$. If $y_1$ lies below $L$ then $\sigma_1$ crosses $\sigma_2$ before the latter reaches $x_3$. By following $\sigma_1$ from $x_1$ till it crosses $\sigma_2$, then $\sigma_2$ to $x_3$, and then $\sigma_3$, we see that we can take $y=y_3\in \mathbf{C}_{A_{n,\epsilon}}^\infty$. Either way, we have found our $y$.

Note that the lines of slope 1 through $x_3$ and $y_2$ are well separated. The closest they can be is when $x_3=(-k, k_3)$ and $y_2=(-k, -\theta_1(n+k))$,  so we always have $\tau_{x_3}^{x_3}\ge -\theta n + k(\theta-\theta_1)-1$. In particular, if $\mathbf{C}_{x_3}$ is finite, then $-\theta n + k(\theta-\theta_1)-1\le \tau_{x_3}^{x_3}<\infty$. But there are at most $n+k$ possible values for $x_3$. Taking a union over these values shows that
\begin{multline*}
1-P(\text{$\exists$ a point $y=(-k,z)\in \mathbf{C}_{A_{n,\epsilon}}^\infty$ that lies above $L$})\\
\le e^{-\gamma_3\epsilon n} + e^{-\gamma_1(k+n)}+C_2(n+k)e^{-\gamma_2(-\theta n + k(\theta-\theta_1)-1)}.
\end{multline*}
In fact, the first excluded event is common to all $k$, so summing over $k$ we get that 
\begin{align*}
&1-P(\cap_{k\ge 0}\{\text{$\exists$ a point $(-k,z)\in \mathbf{C}_{A_{n,\epsilon}}^\infty$ that lies above $L$}\})\\
&\qquad\qquad \le e^{-\gamma_3\epsilon n}+\sum_{k=0}^\infty[e^{-\gamma_1(k+n)}+C_2(n+k)e^{-\gamma_2(-\theta n + k(\theta-\theta_1)-1)}]\\
&\qquad\qquad \le e^{-\gamma_3\epsilon n}+ C[e^{-\gamma_1 n}+ne^{\gamma_2\theta n}] \le Ce^{-\gamma n}
\end{align*}
for some $C$, provided we choose $\gamma<\min(\gamma_1, -\gamma_2\theta, \gamma_3\epsilon)$. 
Under the above event, there are open paths from a single $x_1\in A_{n,\epsilon}$ to each such $(-k,z)$, so we can take the maximum over all these paths and obtain a single infinite path from $x_1$ that lies completely above $L$. In other words, $1-P(\mathcal{A}_{n,\epsilon,\theta})\le Ce^{-\gamma n}$, as required.
\end{proof}

\section{Ballisticity in the elliptic case}
\label{sec:elliptic}

For i.i.d. uniformly elliptic walks, there are a number of abstract conditions that imply ballisticity, starting with Kalikow's condition \cite{K81}. The proof of ballisticity in that context is due to Sznitman and Zerner \cite{SZ99}. In \cite{Sz01}, Sznitman introduces a condition weaker than Kalikow's, but which also implies ballisticity. He called this condition (T), and it is defined using exponential moments of the walk up to the regeneration time $\mc{T}$. Our condition \eqref{eqn:estimategoal} is therefore very similar in character. 

In \cite{Sz02} he formulated weaker conditions ($\text{T}'$) and $(\text{T})_\gamma$ that don't require knowledge of $\mc{T}$, but instead are based on the distributions of the walk prior to exiting from arbitrarily large slabs. \cite{Sz02} shows that  ballisticity holds under ($\text{T}'$), and as well that ($\text{T}'$) is equivalent to what is there called an ``effective condition''. In other words, a condition that can be verified by finding a large but finite box on which it holds. 

Berger, Drewitz and Ram\'irez show in \cite{BDR14} that ($\text{T}'$) is equivalent to $(\text{T})_\gamma$ for each $0<\gamma<1$, and moreover that these conditions are in turn equivalent to a polynomial decay condition $(\text{P})_{\text{M}}$ that is also effective. Uniform ellipticity is relaxed in \cite{CR13} and then further in \cite{BRS16}, where ballisticity is shown under $(\text{P})_{\text{M}}$, for elliptic (but not uniformly elliptic) walks. In those results, all directions have nonzero probability of being chosen, but certain directions are allowed to have probabilities that decay to zero in a controlled way. 

None of the above ballisticity conditions is strictly local, in the sense that it is formulated solely in terms of the law $\mu$ of the environment at a single site. In contrast, our Propositions \ref{prp:E+} and \ref{prp:transverse} together do provide such local conditions. In the uniformly elliptic case, the best known local condition is the following from \cite{K81}
\begin{equation}
\hat \varepsilon_\ell >0\text{ where } \hat\varepsilon_\ell=\inf_{f\in\mathcal{F}}\frac{E_\mu\Big[\frac{\sum_{e\in\mathcal{E}}\gamma(e)\ell\cdot e}{\sum_{e\in\mathcal{E}}\gamma(e)f(e)}\Big]}{E_\mu\Big[\frac{1}{\sum_{e\in\mathcal{E}}\gamma(e)f(e)}\Big]}.
\label{localKalikow}
\end{equation}
Here $\mathcal{F}$ denotes the set of nonzero functions on $\mathcal{E}$ with values in $[0,1]$. In the presence of uniform ellipticity this implies Kalikow's condition and hence ballisticity (note that \cite{Sz02} differentiates between \eqref{localKalikow} and Kalikow's condition by calling the former {\it Kalikow's criterion}). 

We wish to compare \eqref{localKalikow} with our conditions, and understand what it tells us about Example \ref{exa:NE_SW}. Note that $\hat\varepsilon_\ell$ is a lower bound for a quantity $\varepsilon_\ell$ that arises in Kalikow's condition, which in turn is a lower bound for $v\cdot\ell$, so \eqref{localKalikow} implies $v\cdot\ell>0$. 
Since the above results depend on uniform ellipticity, which fails for Example \ref{exa:NE_SW}, we work with the following uniformly elliptic version instead:

\medskip

\begin{EXA}[modified 2-d orthant model]
$\mu_{\epsilon,\delta}$ is 2-valued, with $\mu(\gamma^{\sss(1)})=p$ and $\mu(\gamma^{\sss(2)})=1-p$; $\gamma^{\sss(1)}(e_1)=\gamma^{\sss(1)}(e_2)=\frac{1-\epsilon}{2}$ and
$\gamma^{\sss(1)}(-e_1)=\gamma^{\sss(1)}(-e_2)=\frac{\epsilon}{2}$; 
$\gamma^{\sss(2)}(e_1)=\gamma^{\sss(2)}(e_2)=\frac{\delta}{2}$ and
$\gamma^{\sss(2)}(-e_1)=\gamma^{\sss(2)}(-e_2)=\frac{1-\delta}{2}$.
\end{EXA}

\medskip

In other words, we add back the missing directions to $S(\gamma^{\sss(1)})$ and $S(\gamma^{\sss(2)})$, with probabilities $\epsilon$ and $\delta$ respectively. Let $\ell=e_1+e_2$. We will examine the range of $p$ for which \eqref{localKalikow} holds while letting $\epsilon\downarrow 0$ or $\delta\downarrow 0$. 

By symmetry (i.e.~$\gamma^{\sss(i)}(e_1)=\gamma^{\sss(i)}(e_2)$, $\gamma^{\sss(i)}(-e_1)=\gamma^{\sss(i)}(-e_2)$), we may assume that $f(e_1)=f(e_2)=a$ and $f(-e_1)=f(-e_2)=b$. So $(a,b)\in F=[0,1]^2\setminus\{(0,0)\}$. 
We get that
$$
\hat\varepsilon_\ell =\inf_{(a,b)\in F}\frac{p(1-2\epsilon)[\delta a+(1-\delta)b]-(1-p)(1-2\delta)[(1-\epsilon)a + \epsilon b]}{p[\delta a+(1-\delta)b]+(1-p)[(1-\epsilon)a + \epsilon b]}.
$$
This fraction has the form $\frac{Aa+Bb}{Ca+Db}$, and an elementary calculation shows that $AD-BC=-2p(1-p)(1-\epsilon-\delta)^2\le 0$. From this it follows that the fraction is 
$\downarrow$ in $a$ and $\uparrow$ in $b$, so the infimum occurs at $(1,0)$, giving
$$
\hat\varepsilon_\ell = 
\frac{(p(1-2\epsilon)\delta - (1-p)(1-2\delta)(1-\epsilon))}{(p\delta+(1-p)(1-\epsilon))}.
$$
Restricting attention to the case $\epsilon<\frac12$ and $\delta<\frac12$, \eqref{localKalikow} becomes that
$$
\frac{p}{1-p}>\frac{(1-2\delta)(1-\epsilon)}{(1-2\epsilon)\delta}.
$$
In other words, sending $\epsilon\downarrow 0$ is inconsequential for \eqref{localKalikow}; it only expands the range of $p$ for which \eqref{localKalikow} implies ballisticity in direction $\ell=(1,1)$. But when $\delta\downarrow 0$, the condition becomes increasingly restrictive;  for there to be any $\epsilon\in(0,\frac12)$ for which \eqref{localKalikow} holds, we require $p>1-\frac{\delta}{1-\delta}$.  The right hand side approaches 1 as $\delta \downarrow 0$. 

We can interpret this observation as saying that the absence of arrows $\NE$ in environment $\gamma^{\sss(2)}$ creates an insurmountable obstacle for obtaining ballisticity in direction $\ell=(1,1)$ via condition \eqref{localKalikow}. The barriers $\SW$ are too strong for  \eqref{localKalikow} to handle. One way of interpreting our main result is that we can overcome the presence of strong barriers $\SW$  by strengthening the forward push, i.e.~including sufficiently many sites $\NE$ that don't permit backwards motion.

\section*{Acknowledgements}
Holmes's research was supported in part by the Marsden Fund, administered by RSNZ. Salisbury's research is supported in part by NSERC.


\bibliographystyle{plain}

\end{document}